\documentclass[12pt, oneside, a4paper]{article}

\usepackage{amsmath}
\usepackage{amsfonts}
\usepackage{amssymb}
\usepackage{amsthm,mathrsfs}
\usepackage{enumerate}
\usepackage{graphicx}
\newtheorem{theorem}{Theorem}[section]
\newtheorem{lemma}[theorem]{Lemma}

\theoremstyle{definition}

\title{\textbf{character graphs with diameter three}}
\author{Mahdi Ebrahimi\footnote{ m.ebrahimi.math@ipm.ir}
 \\
 {\small\em  School of Mathematics, Institute for Research in Fundamental Sciences (IPM)},\\{\small\em P.O. Box: 19395--5746, Tehran, Iran}}

\date{}

\begin{document}

\maketitle

\begin{abstract}
For a finite group $G$, let $\Delta(G)$ denote the character graph built on the set of degrees of the irreducible complex characters of $G$. In this paper, we show that if the diameter of $\Delta(G)$ is equal to three, then the complement of $\Delta(G)$ is bipartite. Also in this case, we determine the structure of the character graph $\Delta(G)$.

 \end{abstract}

\section{Introduction}
$\noindent$ Let $G$ be a finite group and $R(G)$ be the solvable radical of $G$. Also let ${\rm cd}(G)$ be the set of all character degrees of $G$, that is,
 ${\rm cd}(G)=\{\chi(1)|\;\chi \in {\rm Irr}(G)\} $, where ${\rm Irr}(G)$ is the set of all complex irreducible characters of $G$. The set of prime divisors of character degrees of $G$ is denoted by $\rho(G)$. It is well known that the
 character degree set ${\rm cd}(G)$ may be used to provide information on the structure of the group $G$. For example, Ito-Michler's Theorem \cite{[C]} states that if a prime $p$ divides no character degree of a finite group $G$, then $G$ has a normal abelian Sylow $p$-subgroup. Another result due to Berkovich \cite{[D]} says that if a prime $p$ divides
 every non-linear character degree of a group $G$, then $G$ is solvable.
  Also in the late 1990s, Bertram Huppert conjectured that the non-abelian simple groups are essentially determined by the set of their character degrees. He verified the conjecture on a case-by-case basis for many non-abelian simple groups, including the Suzuki groups, many of the sporadic simple groups, and a few of the simple groups of Lie type \cite{con}.

A useful way to study the character degree set of a finite group $G$ is to associate a graph to ${\rm cd}(G)$.
One of these graphs is the character graph $\Delta(G)$ of $G$ \cite{[I]}. Its vertex set is $\rho(G)$ and two vertices $p$ and $q$ are joined by an edge if the product $pq$ divides some character degree of $G$.
 When $G$ is a finite group, some interesting results on the character graph of $G$ have been obtained. For instance, in \cite{1}, it was shown that if $G$ is solvable, then the complement of $\Delta(G)$ is bipartite. Sayanjali et.al \cite{reg} proved that the character graph $\Delta(G)$ of a finite group $G$ is a connected regular graph of odd order, if and only if $\Delta(G)$ is complete.
 Also the character graph $\Delta(G)$ of a solvable group $G$ is Hamiltonian if and only if $\Delta(G)$ is a
block with at least 3 vertices \cite{Ha}. We refer the readers to a survey by Lewis \cite{[M]} for results concerning this graph and related topics.

 The diameter of the character graph $\Delta(G)$ of a finite group $G$ is at most 3 (see \cite{[DM]} and \cite{[P]}). When for a finite solvable group $G$, $\Delta(G)$ has exactly diameter 3, Lewis in \cite{[A]} shows that
we may partition $ \rho(G) $ as $ \rho_{1} \cup \rho_{2} \cup
\rho_{3} \cup \rho_{4} $ where no prime in $ \rho_{1} $ is
adjacent to any prime in $ \rho_{3} \cup \rho_{4} $ and no prime
in $ \rho_{4} $ is adjacent to any prime in $ \rho_{1}  \cup
\rho_{2 } $, every prime in $  \rho_{2 } $ is adjacent to some
primes in $ \rho_{3} $ and vice-versa, and $ \rho_{1}  \cup
\rho_{2 } $ and $ \rho_{3} \cup \rho_{4} $ both determine
complete subgraphs of $ \Delta(G)$. If for a finite group $G$, $\rho(G)$ can be partitioned as above, then we say that $\Delta(G)$ is a duke graph. Now we are ready to state our main result. \\

\noindent \textbf{Main Theorem.}  \textit{Let $G$ be a finite group. If the diameter of $\Delta(G)$ is three, then  $\Delta(G)$ is a duke graph and its complement  is bipartite.}

\section{Preliminaries}
$\noindent$ In this paper, all groups are assumed to be finite and all
graphs are simple and finite. For a finite group $G$, the set of prime divisors of $|G|$ is denoted by $\pi(G)$. Also note that for an integer $n\geqslant 1$,  the set of prime divisors of $n$ is denoted by $\pi(n)$.
If $H\leqslant G$ and $\theta \in \rm{Irr}(H)$, we denote by $\rm{Irr}(G|\theta)$ the set of irreducible characters of $G$ lying over $\theta$ and define $\rm{cd}(G|\theta):=\{\chi(1)|\,\chi \in \rm{Irr}(G|\theta)\}$. We frequently use,  Gallagher's Theorem which is corollary 6.17 of \cite{[isa]}.  We begin with Corollary 11.29 of \cite{[isa]}.

\begin{lemma}\label{fraction}
Let $ N \lhd G$ and $\varphi \in \rm{Irr}(N)$. Then for every $\chi \in \rm{Irr}(G|\varphi)$, $\chi(1)/\varphi(1)$ divides $[G:N]$.
\end{lemma}

  \begin{lemma}\label{lw}\cite{[non]}
 Let $p$ be a prime, $f\geqslant 2$ be an integer, $q=p^f\geqslant 5$ and $S\cong \rm{PSL}_2(q)$. If $q\neq9$ and $S\leqslant G\leqslant \rm{Aut}(S)$, then $G$ has irreducible characters of degrees $(q+1)[G:G\cap \rm{PGL}_2(q)]$ and $(q-1)[G:G\cap \rm{PGL}_2(q)]$.
 \end{lemma}

\begin{lemma}\label{good}\cite{[Ton]}
Let $N$ be a normal subgroup of a group $G$ so that $G/N\cong S$, where $S$ is a non-abelian simple group. Let $\theta \in \rm{Irr}(N)$. Then either $\chi (1)/\theta(1)$ is divisible by two distinct primes in $\pi(G/N)$ for some $\chi \in \rm{Irr}(G|\theta)$ or $\theta$ is extendible to $\theta_0\in \rm{Irr}(G)$ and $G/N\cong A_5$ or $\rm{PSL}_2(8)$.
\end{lemma}

Let $\Gamma$ be a graph with vertex set $V(\Gamma)$ and edge set
$E(\Gamma)$. The complement of $\Gamma$ and the induced subgraph of $\Gamma$ on $X\subseteq V(\Gamma)$
 are denoted by $\Gamma^c$ and  $\Gamma[X]$, respectively.  If  $v,w \in V(\Gamma)$ are vertices of a connected component of $\Gamma$, the distance between $v$ and $ w$ is denoted by $d_\Gamma (v,w)$.  We now state some relevant results on character graphs
needed in the next section.

\begin{lemma}\label{chpsl}\cite{[white]}
Let $G\cong \rm{PSL}_2(q)$, where $q\geqslant 4$ is a power of a prime $p$.\\
\textbf{a)}
 If $q$ is even, then $\Delta(G)$ has three connected components, $\{2\}$, $\pi(q-1)$ and $\pi(q+1)$, and each component is a complete graph.\\
\textbf{b)}
 If $q>5$ is odd, then $\Delta(G)$ has two connected components, $\{p\}$ and $\pi((q-1)(q+1))$.\\
i)
 The connected component $\pi((q-1)(q+1))$ is a complete graph if and only if $q-1$ or $q+1$ is a power of $2$.\\
ii)
  If neither of $q-1$ or $q+1$  is a power of $2$, then $\pi((q-1)(q+1))$ can be partitioned as $\{2\}\cup M \cup P$, where $M=\pi (q-1)-\{2\}$ and $P=\pi(q+1)-\{2\}$ are both non-empty sets. The subgraph of $\Delta(G)$ corresponding to each of the subsets $M$, $P$ is complete, all primes are adjacent to $2$, and no prime in $M$ is adjacent to any prime in $P$.
 \end{lemma}

\begin{lemma}\label{ote}\cite{ME}
Let $G$ be a finite group, $R(G)< M\leqslant G$, $S:=M/R(G)$ be isomorphic to $ \rm{PSL}_2(q)$, where for some prime $p$ and positive integer $f\geqslant 1$, $q=p^f$, $|\pi(S)|\geqslant 4$ and $S\leqslant G/R(G)\leqslant \rm{Aut}(S)$. Also let $\theta\in \rm{Irr}(R(G))$. If $\Delta(G)^c$ is not  bipartite,  then $\theta$ is $M$-invariant.
\end{lemma}

\begin{lemma}\label{direct product}\cite{ME}
Assume that $f\geqslant 2$ is an integer, $q=2^f$,  $S\cong \rm{PSL}_2(q)$ and $G$ is a finite group such that $G/R(G)=S$. If $\Delta(G)[\pi(S)]=\Delta(S)$, then $G\cong S\times R(G)$.
\end{lemma}

When $\Delta(G)^c$ is not a bipartite graph, then there exists a useful restriction on the structure of $G$ as follows:

\begin{lemma}\label{cycle} \cite{AC}
Let $G$ be a finite group and $\pi$ be a subset of the vertex set of $\Delta(G)$ such that $|\pi|> 1$ is an odd number. Then $\pi$ is the set of vertices of a cycle in $\Delta(G)^c$ if and only if $O^{\pi^\prime}(G)=S\times A$, where $A$ is abelian, $S\cong \rm{SL}_2(u^\alpha)$ or $S\cong \rm{PSL}_2(u^\alpha)$ for a prime $u\in \pi$ and a positive integer $\alpha$, and the primes in $\pi - \{u\}$ are alternately odd divisors of $u^\alpha+1$ and  $u^\alpha-1$.
\end{lemma}

\begin{lemma} \label{trick}\cite{AC}
Let $G$ be a finite group. Also let $\mathcal{K}$ be any (nonempty) set of normal subgroups of $G$ isomorphic to  $\rm{PSL}_2(u^\alpha)$ or $\rm{SL}_2(u^\alpha)$, where $u^\alpha\geqslant 4$ is a prime power (possibly with different values of $u^\alpha$). Define $K$ as the product of all the subgroups in $\mathcal{K}$ and $C:=C_G(K)$. Then every prime $t$ in $\rho(C)$ is adjacent  in $\Delta(G)$ to all the primes $q$ (different from $t$) in $|G/C|$, with the possible exception of $(t,q)=(2,u)$ when $|\mathcal{K}|=1$, $K\cong \rm{SL}_2(u^\alpha)$ for some $u\neq 2$ and $Z(K)=P^\prime$, $P\in \rm{Syl}_2(C)$. In any case, $\rho(G)=\rho(G/C)\cup \rho(C)$.
\end{lemma}


\section{Proof of Main Theorem}
$\noindent$In this section, we wish to prove our main result.

\begin{lemma}\label{diam}
Let $G$ be a finite group $p,q\in \rho(G)$ and $d_{\Delta(G)}(p,q)=3$. Then  for every $t\in \rho(G)-\{p,q\}$, $t$ is adjacent to $p$ or $q$ in $\Delta(G)$.
\end{lemma}

\begin{proof}
On the contrary, we assume that there exists $t\in \rho(G)-\{p,q\}$ such that the induced subgraph of $\Delta(G)^c$ on $\pi:=\{p,q,t\}$ is a triangle. Then by Lemma \ref{cycle}, $N:=O^{\pi^\prime}(G)=R\times A$, where $A$ is abelian, $R\cong \rm{SL}_2(u^\alpha)$ or $R\cong \rm{PSL}_2(u^\alpha)$ for a prime $u\in \pi$ and a positive integer $\alpha$, and the primes in $\pi - \{u\}$ are alternately odd divisors of $u^\alpha+1$ and  $u^\alpha-1$. Let $M:=NR(G)$. Then $S:=M/R(G)\cong N/R(N)\cong \rm{PSL}_2(u^\alpha)$ is a non-abelian minimal normal subgroup of $G/R(G)$. Note that $\pi\subseteq \pi(S)$. Let $C/R(G)=C_{G/R(G)}(M/R(G))$. We claim that $C=R(G)$. Suppose on the contrary that $C\neq R(G)$ and let $L/R(G)$ be a chief factor of $G$ with $L\leqslant C$. Then $L/R(G)\cong T^k$, for some non-abelian simple group $T$ and some integer $k\geqslant 1$.
 As $L\leqslant C$, $LM/R(G)\cong L/R(G)\times M/R(G) \cong S\times T^k$. Since $2\in \pi (S)\cap \pi(T) $, $2$ is adjacent to all vertices in $\pi$ which is impossible as $d_{\Delta(G)}(p,q)=3$. Therefore $G/R(G)$ is an almost simple group with socle $S=M/R(G)$. Now one of the following cases occurs:\\
 Case 1. $t=u$. Then as $p,q\in \pi(u^{2\alpha}-1)$ and $d_{\Delta(G)}(p,q)=3$, $t=2$.
  We claim that $G=M$. On the contrary, suppose $G\neq M$. Then using Lemma \ref{lw}, $[G:M](u^\alpha\pm 1)\in \rm{cd}(G/R(G))\subseteq \rm{cd}(G)$.
   It is a contradiction as $d_{\Delta(G)}(p,q)=3$.
    Hence $G=M$. Let $x,y\in \pi(S)$ be two distinct primes adjacent in $\Delta(G)$ and non-adjacent in $\Delta(S)$. Then $|\pi(S)|\geqslant 4$ and for some $\chi\in \rm{Irr}(G)$, $xy|\chi(1)$.
    Suppose $\theta \in \rm{Irr}(R(G))$ is a constituent of $\chi_{R(G)}$. Using Lemma \ref{ote}, $\theta$ is $G$-invariant.
     Hence as the Schur multiplier of $S$ is trivial, by Gallagher's Theorem, $\rm{cd}(G|\theta)=\{\theta(1)m\mid m\in \rm{cd}(S)\}$. Hence as $\chi(1)\in \rm{cd}(G|\theta)$, $\theta(1)$ is divisible by $x$ or $y$. It is a contradiction as $p,q\in \pi(u^{2\alpha}-1)$ and $d_{\Delta(G)}(p,q)=3$. Thus $\Delta(G)[\pi(S)]=\Delta(S)$. Hence by Lemma \ref{direct product}, $G\cong S\times R(G)$ and the diameter of $\Delta(G)$ is at most $ 2$. It is a contradiction as $d_{\Delta(G)}(p,q)=3$.\\
Case 2. $u\neq t$. Then without loss of generality, we can assume that $u=p$.
There exists $\epsilon\in \{\pm 1\}$ such that $q\in \pi(u^\alpha+\epsilon)$.
 Let $b$ be a prime divisor of $[G:M]$.
 We claim $b$ is adjacent to $q$.
 If $b\neq 2$, then by Lemma \ref{lw}, we are done.
 Hence we suppose $b=2$. If $b=p$, then using Lemma \ref{lw}, $p$ and $q$ are adjacent vertices in $\Delta(G)$ which is impossible. Therefore $p$ is odd and it is clear that $b$ is adjacent to $q$ in $\Delta(G)$. Thus as $d_{\Delta(G)}(p,q)=3$, there exists a prime $x\in \rho(G)$ such that $x\nmid [G:M]$ and in $\Delta(G)$, $x$ is adjacent to $p$. There exists $\chi \in \rm{Irr}(G)$ such that $px|\chi(1)$.
  Now let $\varphi \in \rm{Irr}(M)$ and $\theta\in \rm{Irr}(R(G))$ be constituents of $\chi_M$ and $\varphi_{R(G)}$, respectively. Then by Lemma \ref{fraction}, $px|\varphi(1)$. Since $\Delta(G)^c[\pi]$ is  a triangle,  using Lemmas \ref{good} and \ref{ote}, $\theta$ is $M$-invariant.
   Thus as $\rm{SL}_2(u^\alpha)$ is the Schur representation of $S$, for some $\lambda\in \rm{Irr}(Z(\rm{SL}_2(u^\alpha)))$,  $\rm{cd}(M|\theta)=\{\theta(1)m\mid m\in \rm{cd(SL}_2(u^\alpha)|\lambda)\}$.
    Hence as $\varphi(1)\in \rm{cd}(M|\theta)$, $\theta(1)$ is divisible by $x$ or $p$.
    Therefore $q$ is adjacent to $x$ or $p$ which is impossible as $d_{\Delta(G)}(p,q)=3$.
\end{proof}

\noindent Proof of Main Theorem.
Since the diameter of $\Delta(G)$ is equal to $3$, there exist $p,q\in \rho(G)$ such that $d_{\Delta(G)}(p,q)=3$. Now let
$$\rho_1:=\{x\in \rho(G)\mid d_{\Delta(G)}(x,p)=3\},$$
$$\rho_2:=\{x\in \rho(G)\mid d_{\Delta(G)}(x,p)=2\},$$
$$\rho_3:=\{x\in \rho(G)\mid d_{\Delta(G)}(x,q)=2\},$$
$$\rho_4:=\{x\in \rho(G)\mid d_{\Delta(G)}(x,q)=3\}.$$
Using Lemma \ref{diam}, it is easy to see that $\rho(G)$ can be partitioned as $\rho(G)=\rho_1\cup \rho_2\cup \rho_3\cup \rho_4$,  where no prime in $ \rho_{1} $ is
adjacent to any prime in $ \rho_{3} \cup \rho_{4} $ and no prime in $ \rho_{4} $ is adjacent to any prime in $ \rho_{1}  \cup \rho_{2 } $, every prime in $  \rho_{2 } $ is adjacent to some primes in $ \rho_{3} $ and vice-versa, every prime in $\rho_2$ is adjacent to all vertices in $\rho_1$, every prime in $\rho_3$ is adjacent to all vertices in $\rho_4$, and  the induced subgraphs of $ \Delta(G)$ on $ \rho_{1}$ and $ \rho_{4} $ are complete. Now we claim that $\Delta(G)[\rho_2]$  is a complete graph. On the contrary, we assume that $p_1,p_2\in \rho_2$ are non-adjacent vertices in $\Delta(G)$. Then the induced subgraph of $\Delta(G)^c$ on $\pi:=\{p_1,p_2,p\}$ is a triangle. Thus using Lemma \ref{cycle}, $N:=O^{\pi^\prime}(G)=T\times A$, where $A$ is abelian, $T\cong \rm{SL}_2(u^\alpha)$ or $T\cong \rm{PSL}_2(u^\alpha)$ for a prime $u\in \pi$ and a positive integer $\alpha$, and the primes in $\pi - \{u\}$ are alternately odd divisors of $u^\alpha+1$ and  $u^\alpha-1$. Let $C:=C_G(N)$ and $M:=NC$. If $2\in \rho(C)$ and when $p\neq 2$, $2$ and $p$ are adjacent vertices in $\Delta(G)$, or $\rho(C)$ contains an odd prime, then by Lemma \ref{trick}, $d_{\Delta(G)}(p,q)\leqslant 2$ which is impossible. Hence either $C$ is abelian, or $\rho(C)=\{2\}$, $p\neq 2$ and $2$ is not adjacent to $p$ in $\Delta(G)$. Thus $C=R(G)$ and $G/R(G)$ is an almost simple group with socle $S:=M/R(G)\cong \rm{PSL}_2(u^\alpha)$. Since $\rho_3\neq \emptyset$, there exists $t\in \rho_3$ such that $t$ and $p$ are adjacent vertices in $\Delta(G)$. If $p$ or $t$ divides $[G:M]$, then using Lemma \ref{lw} and this fact that $q\neq u$, we can see that $d_{\Delta(G)}(p,q)\leqslant 2$ which is a contradiction. Hence $t,p\notin \pi([G:M])$. There exists $\chi\in \rm{Irr}(G)$ so that $pt|\chi(1)$. Let $\varphi\in \rm{Irr}(M)$ and $\theta\in \rm{Irr}(R(G))$ be constituents of $\chi_M$ and $\varphi_{R(G)}$, respectively. By Lemma \ref{fraction}, $pt|\varphi(1)$.  Since $\Delta(G)^c[\pi]$ is  a triangle,  using Lemmas \ref{good} and \ref{ote}, $\theta$ is $M$-invariant. Thus as $\rm{SL}_2(u^\alpha)$ is the Schur representation of $S$ and $\varphi \in \rm{cd}(M|\theta)$, we deduce that $\theta$ is linear and for some $\epsilon\in \{\pm 1\}$, $t,p\in \pi(u^\alpha+\epsilon)$. If $q|[G:M]$, then using Lemma \ref{lw}, we deduce that $d_{\Delta(G)}(p,q)=1$ and it is a contradiction. Hence as $q\neq u$, $q\in \pi(u^\alpha-\epsilon)$.
 Therefore using  Lemma \ref{lw} and this fact that $d_{\Delta(G)}(p,q)=3$, we can see that $G=M$ and $u=2\in \rho_2$. Thus for some $\chi\in \rm{Irr}(G)$, $2q|\chi(1)$. Suppose $\theta\in \rm{Irr}(R(G))$ is a constituent of $\chi_{R(G)}$. Then as the Schur multiplier of $S$ is trivial, Lemma \ref{ote} implies that $\theta$ is extendible to $G$. Thus by Gallagher's Theorem, $\chi(1)\in \rm{cd}(G|\theta)=\{\theta(1)m\mid m\in \rm{cd}(S)\}$. Hence as $2q|\chi(1)$,  $\pi(\theta(1))=\{2\}$ and $2$ is adjacent to $p$ in $\Delta(G)$ which is impossible. Thus $\Delta(G)[\rho_2]$ is a complete graph. Similarly, we can see that $\Delta(G)[\rho_3]$ is too and the proof is completed.


\section*{Acknowledgements}
This research was supported in part
by a grant  from School of Mathematics, Institute for Research in Fundamental Sciences (IPM).


\end{document}